\documentclass[reqno]{amsart}%
\usepackage{enumerate}
\usepackage{amsfonts}
\usepackage{amsmath}
\usepackage{amssymb}
\usepackage{cases}
\usepackage{hyperref}
\usepackage{float}
\usepackage{graphicx}%
\newtheorem{theorem}{Theorem}[section]

\newtheorem{proposition}[theorem]{Proposition}

\theoremstyle{definition}

\newtheorem{remark}[theorem]{Remark}

\newcommand{\norm}[1]{\left\Vert#1\right\Vert}

\numberwithin{equation}{section}

\hypersetup{
colorlinks,
citecolor=blue,
filecolor=black,
linkcolor=blue,
urlcolor=magenta
}
\hypersetup{linktocpage}
\begin{document}
\title[Interpolation scattering for wave equations]{Interpolation scattering for wave equations with singular potentials and singular data}

%\author[\ T.T. Ngoc ]{\ \ Tran Thi Ngoc* }
%\address{Tran Thi Ngoc \hfill\break *Corresponding author, Faculty of Fundamental Sciences, East Asia University of Technology, Polyco building, Trinh Van Bo, Nam Tu Liem, Hanoi, Viet Nam}
%\email{ngoctt@eaut.edu.vn}
\author[\ T.X. Pham]{\ \ Truong Xuan Pham}
\address{Truong Xuan Pham\hfill\break
Thang Long Institute of Mathematics and Applied Sciences (TIMAS), Thang Long University, Nghiem Xuan Yem,
Hanoi, Vietnam\hfill\break
and Faculty of Mathematics and Informatics, Hanoi University of Science and Technology,
1 Dai Co Viet, Hanoi, Vietnam} 
\email{xuanpt@thanglong.edu.vn and phamtruongxuan.k5@gmail.com}

\thanks{{\bf Acknowledgment.} P.T. Xuan was funded by the Postdoctoral Scholarship Programme of Vingroup Innovation Foundation (VINIF), code VINIF.2023.STS.55.}

\begin{abstract}
In this paper we investigate a construction of scattering for wave-type equations with singular potentials on the whole space $\mathbb{R}^n$ in a framework of weak-$L^p$ spaces. First, we use an Yamazaki-type estimate for wave groups on Lorentz spaces and fixed point arguments to prove the global well-posedness for wave-type equations on weak-$L^p$ spaces. Then, we provide a corresponding scattering results in such singular framework. Finally, we use also the dispersive estimates to establish the polynomial stability and improve the decay of scattering in {weak-$L^p$ spaces}.
\vspace{0.2cm}

 \vspace{0.2cm}


\end{abstract}

\maketitle

%\tableofcontents

\font\nho=cmr10

%\def\v{\mathbf V}

%\pagestyle{plain}

%\def\v{\mathbf V}

%\def\@frontmatterwidth{16cm}
%\textheight=22cm
%\oddsidemargin=0cm
%\evensidemargin=0cm
%\begin{frontmatter}

\section{Introduction}
In the present paper, we are concerned with the wave-type equation on the whole space $\mathbb{R}^{n}$ $(n\geqslant 5)$:
\begin{equation}%
\begin{cases}
\partial_{t}^{2}u(t,x)-\Delta_{x}u(t,x)+V_1(x)u(t,x) = V_2(x)F(u(t,x)),\\
u(0,x)=u_{0}(x),\,\,\partial_{t}u(t,x)=u_{1}(x),
\end{cases}
\label{OKlein}%
\end{equation}
where $\Delta_{x}$ stands for the Laplace-Beltrami operator, the potentials $V_1(x) = \dfrac{c_1}{|x|^2}$ (is called Hardy potential), $V_2(x)=\dfrac{c_2}{|x|^b}\, (\text{where   }0<b<2)$ and the
nonlinearity $F(u)$ satisfies
\begin{equation}
|F(0)|=0\text{ and }|F(u)-F(v)|\leqslant C(|u|^{q-1}+|v|^{q-1})|u-v|,\text{ for
}q>1. \label{nonlinearity-1}%
\end{equation}
The wave equation with generalized potential $V_1(x)$ in the linear term was considered in many works (see for examples \cite{Fe2017,An,Cu,Ge}). 
Moreover, the nonlinear part formed by $V_2(x)|u|^{q-1}u$ was studied in some previous works for nonlinear wave equations (see \cite{Luce}) with $V_2(x)$ is regular and for nonlinear Schr\"odinger equations (see for examples \cite{Campos,Farah1,Farah2,Ahmed,
Guli2020,Ou2025}) with $V_2(x)$ is singular.

The wave equations on Euclidean space $\mathbb{R}^n$ have been extensively to study for a long time by many authors. We recall briefly some important results. The existence of global solutions was
analyzed by Georgiev \textit{et al.} \cite{Ge1997}, Ginibre and Velo
\cite{Gi1}, Zhou \cite{Zho1995}, Belchev \textit{et al.} \cite{Bel}, among
others. The nonexistence of global solutions was studied by Sideris \cite{Si}.
Time decays of solutions for wave equations were proved in \cite{Fec1982,Gi2,G3,Str1}. In particular, Fecher \cite{Fec1982,Fec1984} obtained results on
time decays and established the nonlinear small data scattering for wave and
Klein-Gordon equations in $\mathbb{R}^{3}$. The scattering theory was also
studied by using the time decays in the works \cite{G4, Lin1995, Str2}. After
that, Hidano \cite{Hi1998,Hi2000,Hi2001} obtained the small data scattering
and blow-up theory for nonlinear wave equation on $\mathbb{R}^{n}$ with
$n=3,\,4$ by using the integral representation formula of solutions.
Concerning the nonlinear wave equations, Strichartz estimates were obtained in
\cite{Ge1997,G4,Tao}. Moreover, we quote the results on blow-up \cite{Gl,Jo}
and the life-span of solutions \cite{LiYu,LiZho}.

The program of study the well-posedness and scattering in weak-$L^{p}$ spaces for some dispersive equations (such as Schr\"odinger, wave and Boussinesq equations) was initially introduced by Cazenave et al. in \cite{Ca2001} for Schr\"{o}dinger equations on $\mathbb{R}^{n}$ by using Strichartz-type estimates and establishing the existence in mixed space-time weak-$L^{p}$ spaces, i.e., $L^{p,\infty}(\mathbb{R}\times\mathbb{R}^{n})$ with $p=\frac
{(b-1)(N+2)}{2}$. Then, Ferreira et al. gave another approach in \cite{Fe2009} to obtain the global well-posedness and asymptotic behaviour for Schr\"{o}dinger equations on $\mathbb{R}^{n},$ but using
dispersive-type estimates in a
framework of time polynomial weighted spaces based on the $L^{(p,\infty
)}(\mathbb{R}^{n})$ with $p=b+1$. Their results extend the ones obtained in the
$L^{p}$-setting by Cazenave and Weissler \cite{Ca1998, Ca2000}. After that, the well-posedness and scattering results for Boussinesq equations \cite{Fe2011} and wave equations have obtained in \cite{Fe2017} and \cite{Liu2009}, respectively. 
The method used in \cite{Fe2017,Fe2009,Liu2009} bases on the $L^{(p,r)}$-$L^{(p^{\prime},r)}$-dispersive estimates on $\mathbb{R}^{n},$ where $1/p+1/p^{\prime}=1$, $1\leqslant r\leqslant\infty,$ and $L^{(p,r)}$ stands for the so-called Lorentz space. The weak-$L^{p}$ space corresponds to the case $r=\infty.$ The global well-posedness and scattering in the space $L^{(p,\infty)}(\mathbb{R}^{n})$ via that approach require the use of suitable Kato-type classes. In addition, the existence of asymptotically almost periodic mild solutions for wave equations within the same framework was established in \cite{SacXuan2025}.

Let us present more details about the works \cite{Fe2017,Liu2009} for wave equations. In particular, Liu  \cite{Liu2009} used the $L^p-L^{p'}$-dispersive estimates of wave operators to establish the global well-posedness of mild solution for the scalar wave equations $\partial_t^2u -\Delta u = \lambda|u|^{b-1}u$ on $\mathbb{R}^n$ in the following time polynomial weighted space
\begin{equation}
\mathcal{E}_{\alpha}=\left\{  u\text{ measurable}:\text{ }\sup_{t\in
\mathbb{R}}\text{ }|t|^{\frac{\alpha}{2}}\left\Vert u(t,\cdot)\right\Vert _{L^{(b+1,\infty
)}}<\infty\right\}
\end{equation}
as well as the initial-data class $\left\{  (u_0,u_1)\in \mathcal{S}^\prime\times\mathcal{S}^{\prime
}:\text{ }\sup\limits_{t\in\mathbb{R}}\text{ }|t|^{\frac{\alpha}{2}}\left\Vert \dot{W}(t)u_0 + W(t)u_1
\right\Vert _{L^{(b+1,\infty)}}<\infty\right\}  $, where $W(t)$ is the wave operator and the power $\alpha = \dfrac{4}{b-1} - \dfrac{2n}{b+1}$ for suitable numbers $b$ and $n$ (for details see \cite[Theorem 2.9]{Liu2009}). Besides, Ferreira et al.  considered in \cite{Fe2017} the wave-type equations (with only a Hardy potential) that have form \eqref{OKlein} for $V_1(x)=\dfrac{c}{|x|^2}$ and $V_2(x)$ is a constant. Using $L^{l_1}-L^{l_2}$-dispersive estimates for the wave groups (see inequalities \eqref{W2} and \eqref{W3} below), the authors established a Yamazaki-type estimate for wave groups, then using this estimate to proven the global-in-time well-posedness for such wave-type equations on a weak-Lorentz space $L^\infty(\mathbb{R},L^{(r_0,\infty)}_{rad}(\mathbb{R}^n))$ which does not consist the time weight factor as in \cite{Liu2009} (for more details see \cite[Theorem 3.3]{Fe2017}). On hand, the works \cite{Fe2017,Liu2009} provide the global well-posedness for wave equations with the initial data outside the usual energy space $H^1(\mathbb{R}^n)\times L^2(\mathbb{R}^n)$. On the other hand, the work \cite{Fe2017} shows the useful role of weak-$L^p$ spaces to treating the singular potentials.

In the present paper, we develope the method used in \cite{Fe2017} to obtain the global well-posedness, an interpolation scattering and a polynomial stability for wave-type equations \eqref{OKlein}. In particular, we employ the Yamazaki-type estimate of the wave group and fixed point arguments to establish the global well-posedness and a scattering results for equation \eqref{OKlein} in the framework of weak-$L^p$ spaces (see Theorem \ref{Global}). Then, we use the $L^{l_1}-L^{l_2}$-dispersive estimates of the wave groups to prove the polynomial stability with the initial data in such singular framework (see Theorem \ref{polynomial}). Using the stability we improve the decay of scattering behaviour (see Remark \ref{rem}).
Since, our scattering result is obtained only on the weak-$L^p$ space which does not consists the time weighted factors, we call this scattering-type by the terminology \textit{''interpolation scattering''}.

The paper is organized as follows. In Subsection \ref{S2}, we provide the wave-type equations with singular potentials on $\mathbb{R}^{n}$ and the dispersive estimates in the $L^{p}
$-setting. In Subsection \ref{S3}, we recall the notion of Lorentz spaces
$L^{(p,q)}$ on $\mathbb{H}^{n}$ and the dispersive and Yamazaki-type estimates of the wave operators on the $L^{(p,q)}$ spaces. We state and prove the main results consisting the global-in-time well-posedness and interpolation scattering in Subsection \ref{S31} and polynomial stability in Subsection \ref{S32}.

\section{Preliminaries}
\subsection{Wave-type equations and dispersive estimates}\label{S2} 
We consider the semilinear wave equation with singular potentias on $\mathbb{R}^{n}$, where $n\geqslant 5$:
\begin{equation}%
\begin{cases}
\partial_{t}^{2}u(t,x)-\Delta_{x}u(t,x)+V_1(x)u(t,x) = V_2(x)F(u(t,x)),\\
u(0,x)=u_{0}(x),\,\,\partial_{t}u(t,x)=u_{1}(x),
\end{cases}
\label{Klein}%
\end{equation}
where $V_1(x)=\dfrac{c_1}{|x|^2}$, $V_2(x) = \dfrac{c_2}{|x|^b}\,\, (\text{where   } 0<b<2)$ and the nonlinearity $F(u(t,x))$
satisfies
\begin{equation}
F(0)=0\text{ and }|F(u)-F(v)|\leqslant C(|u|^{q-1}+|v|^{q-1})|u-v|
\label{Klein-Cond-F}%
\end{equation}
with $q>1$. 

Setting $D=\sqrt{-\Delta_x}$ and by using Duhamel's principle we can define the mild solution of equation (\ref{Klein}) as follows
\begin{equation}
u(t)=\dot{W}(t)u_{0}+W(t)u_{1}+\int_{0}^{t}W(t-s) \left(  -V_1(x)u(s,x) + V_2(x)F(u(s,x)) \right) ds,
\label{intergralEq}%
\end{equation}
where $W(t)$ is the wave group and $\dot{W}(t)=\partial_tW(t)$. Recall that
\begin{equation}
W(t)=\dfrac{\sin(tD)}{D}\text{ \ and \ }\dot{W}(t)=\cos(tD).
\label{aux-wave-rel-1}%
\end{equation}

To recall the dispersive estimates of the wave group $\left\{ W(t)\right\}_{t\in \mathbb{R}}$ which are useful to establish the well-posedness in next sections, we define the following points
\begin{eqnarray}
&&P_1 = \left( \frac{1}{2}+\frac{1}{n+1}, \frac{1}{2} - \frac{1}{n+1}  \right), \,\,\, P_2 = \left( \frac{1}{2} - \frac{1}{n-1}, \frac{1}{2} - \frac{1}{n-1}  \right)\cr
&&P_3 = \left( \frac{1}{2}+\frac{1}{n-1}, \frac{1}{2} + \frac{1}{n-1}  \right),\,\,\, P_4 = \left( 1, \frac{n-1}{2n} \right) \hbox{   and    } P_5=(1,1).
\end{eqnarray}
We know that the wave group $\left\{ W(t)\right\}_{t\in \mathbb{R}}$ is bounded from $L^{l_1}$ to $L^{l_2}$ at $t=1$ (see \cite{Bre,Peral,Strich}):
\begin{equation}\label{W1}
\|W(1)h\|_{l_2} \leqslant M_1 \| h\|_{l_1}
\end{equation}
provided that $\left( \dfrac{1}{l_1},\dfrac{1}{l_2} \right) \in \overline{\Delta_{P_1P_2P_3}}$, where $\overline{\Delta_{P_1P_2P_3}}$ is the closure of $\Delta_{P_1P_2P_3}$. From scaling properties of $W(t)$ and $L^p$-spaces we obtain from inequality \eqref{W1} that
\begin{equation}\label{W2}
\|W(t)h\|_{l_2} \leqslant M_1 |t|^{-n\left(  \frac{1}{l_1} - \frac{1}{l_2} \right) +1} \| h\|_{l_1}
\end{equation}
If $h$ is radial symmetric, then one can extend the range of \eqref{W1} to $\overline{\Delta_{P_2P_4P_5}}$ except for the semi-open segment line $]P_1,P_4]$.
Precisely, if $\left( \dfrac{1}{l_1},\dfrac{1}{l_2} \right) \in {\Delta_{P_2P_4P_5}}$ and $h\in L^{l_1}_{rad}(\mathbb{R}^n)$, then we have
\begin{equation}\label{W3}
\|W(t)h\|_{l_2} \leqslant M_2 |t|^{-n\left(  \frac{1}{l_1} - \frac{1}{l_2} \right) +1} \| h\|_{l_1}.
\end{equation}
The inequalities \eqref{W2} and \eqref{W3} are called the $L^{l_2}-L^{l_1}$-dispersive estimates of the wave group $\left\{W(t)\right\}_{t\in \mathbb{R}}$.

%Moreover, we recall also the $L^p-L^{p'}$-dispersive estimates for the wave group $\left\{W(t)\right\}_{t\in \mathbb{R}}$ as follows (see \cite{Miao,Plan}): let $p\in \left[2,\dfrac{2(n+1)}{n-1}\right]$ and $p'$ such that $\dfrac{1}{p}+ \dfrac{1}{p'}=1$. There exists a positive constant $L>0$ such that
%\begin{equation}\label{LpLp'}
%\| W(t)h\|_p \leqslant L |t|^{1-n\left(1- \frac{2}{p}  \right)} \| h\|_{p'}
%\end{equation}
%for all $h \in L^{p'}(\mathbb{R}^n)$.
\subsection{Lorentz spaces and Yamazaki-type estimates for wave groups}\label{S3}

Let $\Omega$ be a subset of $\mathbb{R}^{n}$. For $0<p<\infty,$ denote by
$L^{p}(\Omega)$ the space of all $L^{p}$-integrable functions on $\Omega.$ Let
$1<p_{1}<p_{2}\leqslant\infty,$ $\theta\in(0,1)$ and $1\leqslant
r\leqslant\infty$ with $\dfrac{1}{p}=\dfrac{1-\theta}{p_{1}}+\dfrac{\theta
}{p_{2}}.$ The Lorentz space $L^{(p,r)}(\Omega)$ is defined as the
interpolation space $\left(  L^{p_{1}},L^{p_{2}}\right)  _{\theta,r}=L^{(p,r)}
$ with the natural norm $\left\Vert \cdot\right\Vert _{(p,r)}$ induced by the
functor $(\cdot,\cdot)_{\theta,r}.$ In particular, $L^{p}(\Omega
)=L^{(p,p)}(\Omega)$ and $L^{(p,\infty)}$ is the so-called weak-$L^{p}$ space
or the Marcinkiewicz space on $\Omega$.

For $1\leqslant q_{1}\leqslant p\leqslant q_{2}\leqslant\infty,$ we have the
following relation
\begin{equation}
L^{(p,1)}(\Omega)\subset L^{(p,q_{1})}(\Omega)\subset L^{p}(\Omega)\subset
L^{(p,q_{2})}(\Omega)\subset L^{(p,\infty)}(\Omega). \label{Inclusion}%
\end{equation}
Let $1<p_{1},p_{2},p_{3}\leqslant\infty$ and $1\leqslant r_{1},r_{2}%
,r_{3}\leqslant\infty$ be such that $\dfrac{1}{p_{3}}=\dfrac{1}{p_{1}}%
+\dfrac{1}{p_{2}}$ and $\dfrac{1}{r_{1}}+\dfrac{1}{r_{2}}\geqslant\dfrac
{1}{r_{3}}$. We have the H\"{o}lder inequality%
\begin{equation}
\left\Vert fg\right\Vert _{{(p_{3},r_{3})}}\leqslant C\left\Vert f\right\Vert
_{{(p_{1},r_{1})}}\left\Vert g\right\Vert _{{(p_{2},r_{2})}}, \label{Holder}%
\end{equation}
where $C>0$ is a constant independent of $f$ and $g$. Moreover, for
$1<p_{1}<p_{2}\leqslant\infty$, $0<\theta<1$, $1\leqslant r_{1},r_{2}%
,r\leqslant\infty,$ and $\dfrac{1}{p}=\dfrac{1-\theta}{p_{1}}+\dfrac{\theta
}{p_{2}},$ by reiteration theorem (see \cite[Theorem 3.5.3 ]{BeLo}), we have
the interpolation property
\begin{equation}
\left(  L^{(p_{1},r_{1})}(\Omega),L^{(p_{2},r_{2})}(\Omega)\right)  _{\theta,r}=L^{(p,r)}(\Omega).
\label{interp1}%
\end{equation}

By using interpolation argument (see \cite{BeLo}) we can extend the dispersive estimates \eqref{W2} and {\eqref{W3} to the
framework of Lorentz spaces. In particular, we have
\begin{equation}\label{W4}
\|W(t)h\|_{(l_2,z)} \leqslant M_1 |t|^{-n\left(  \frac{1}{l_1} - \frac{1}{l_2} \right) +1} \| h\|_{(l_1,z)}
\end{equation}
for $\left( \dfrac{1}{l_1},\dfrac{1}{l_2} \right) \in \overline{\Delta_{P_1P_2P_3}}$ and $h \in L^{(l_1,z)}(\mathbb{R}^n)\,\, (1\leqslant z \leqslant \infty)$. The following inequality hold also
\begin{equation}\label{W5}
\|W(t)h\|_{(l_2,z)} \leqslant M_2 |t|^{-n\left(  \frac{1}{l_1} - \frac{1}{l_2} \right) +1} \| h\|_{(l_1,z)}
\end{equation}
for $\left( \dfrac{1}{l_1},\dfrac{1}{l_2} \right) \in {\Delta_{P_2P_4P_5}}$ and $h\in L^{(l_1,z)}_{rad}(\mathbb{R}^n)\,\, (1\leqslant z \leqslant \infty)$. 

Below, we recall an Yamazaki-type inequality for the wave group which was proven in \cite{Fe2017}.
\begin{proposition}\label{Yamazaki}(Yamazaki-type estimate).
Let $f$ be radially symmetric and $n\geqslant 3$ odd. If $1<d_1,d_2 < \dfrac{2(n-1)}{n-1}\,\, (\infty \text{  if  } n=3)$ with $\left( \dfrac{1}{d_1},\dfrac{1}{d_2}  \right) \in \Delta_{P_2P_4P_5}$, then $|t|^{n\left( \frac{1}{d_1}-\frac{1}{d_2}\right)-2}W(t)f \in L^1(\mathbb{R}, L^{(d_2,1)}(\mathbb{R}^n))$ and there is a positive constant $C$ such that
\begin{equation}
\int_{\mathbb{R}}|t|^{n\left( \frac{1}{d_1}-\frac{1}{d_2}\right)-2} \|W(t)f\|_{(d_2,1)}dt \leqslant C\| f\|_{(d_1,1)}
\end{equation}
for all $f\in L^{(d_1,1)}_{rad}(\mathbb{R}^n)$.
\end{proposition}
\begin{proof}
The proof of this proposition was given in \cite[Lemma 4.1]{Fe2017} by using inequality \eqref{W5} and interpolation arguments.
\end{proof}

\section{Well-posedness, interpolation scattering and stability on weak-$L^p$ spaces}\label{S4}

\subsection{Global well-posedness and interpolation scattering}\label{S31}
In this part we state and prove the global well-posedness and scattering for equation \eqref{Klein} on the framework of weak-$L^p$ spaces in the following theorem.
\begin{theorem}\label{Global}
Let $n\geqslant 5$ be old,  $p=\dfrac{2q-b}{2-b} \geqslant \dfrac{n^2+n-4}{n(n-3)}$ and $r_0 = \dfrac{n(p-1)}{2}$. Suppose $(u_0,u_1)\in \mathcal{I}_{rad}$, where
\begin{equation*}
\mathcal{I}_{rad} = \left\{ (u_0,u_1) \in \mathcal{S}'_{rad}\times \mathcal{S}'_{rad}: \dot{W}(\cdot)u_0 + W(\cdot)u_1 \in L^\infty(\mathbb{R}, L^{(r_0,\infty)}_{rad}(\mathbb{R}^n)) \right\},
\end{equation*}
the subindex ''rad'' means space of radial distributions.
\begin{itemize}
\item[$(i)$] (Global well-posedness). If $\| (u_0,u_1)\|_{\mathcal{I}_{rad}}$ and $c_1$ are small enough, then equation \eqref{Klein} has a unique mild solution $u$  in a ball with small radius in $L^\infty(\mathbb{R},L^{(r_0,\infty)}_{rad}(\mathbb{R}^n))$ provided by formula \eqref{intergralEq}.

\item[$(ii)$] (Interpolation scattering). There exist $(u_{0}^{\pm},u_{1}^{\pm})\in\mathcal{I}_{rad}$ satisfying
\begin{equation}
\left\Vert u(t)-u^{+}(t)\right\Vert _{(r_0,\infty)} \longrightarrow 0,\hbox{
as  }t\rightarrow+\infty,\label{ine16}%
\end{equation}
\begin{equation}
\left\Vert u(t)-u^{-}(t)\right\Vert _{(r_0,\infty)} \longrightarrow 0,\hbox{  as  }t\rightarrow-\infty,\label{ine17}%
\end{equation}
where $u^{\pm}$ are the unique solutions of the associated linear problem
$u^{\pm}=\dot{W}(t)u_{0}^{\pm}+W(t)u_{1}^{\pm}$ that are called ''interpolation scattering'' states of the mild solution $u$ at the future time and past time, respectively.
\end{itemize}
\end{theorem}

\begin{proof}
\item[$(i)$] For a given $v \in L^\infty(\mathbb{R},L^{(r_0,\infty)}_{rad}(\mathbb{R}^n))$, we consider the following linear integral equation
\begin{equation}\label{InterEq}
u(t)=\dot{W}(t)u_{0}+W(t)u_{1}+\int_{0}^{t}W(t-s) \left(  -V_1(x)v(s,x) + V_2(x)F(v(s)) \right) ds.
\end{equation}
Setting 
\begin{equation}
\zeta(f)(t,x) = \int_0^t W(t-s)f(s,x)ds.
\end{equation}
We have
\begin{eqnarray}
\zeta(f)(t,x) &=& \int_{-\infty}^{+\infty} \omega(t-s,x-y)f(s,y)dyds,
\end{eqnarray}
where the kernel $\omega$ is determined by
\begin{equation}
\widehat{\omega}(t-s,\xi)= \begin{cases}
\frac{\sin((t-s)|\xi|)}{|\xi|},\,\,\, &0<s<t\\
0,&\text{    otherwise}.
\end{cases}
\end{equation}

Let $s=\dfrac{r_0}{p}$, we have
\begin{equation}
\frac{1}{s}-\frac{1}{r_0} = \frac{p-1}{r_0} =  \frac{2}{n}.
\end{equation}
Hence, 
\begin{equation}\label{d1d2}
\frac{n}{r_0'}-\frac{n}{s'} - 2 =0,
\end{equation}
where $r_0' = \dfrac{r_0}{r_0-1}$ and $s'=\dfrac{s}{s-1}$. Moreover, since $p>\dfrac{n^2+n-4}{n(n-3)}$, we obtain that
\begin{equation}
\left( 1 - \frac{2}{n(p-1)}, 1-\frac{2p}{n(p-1)}  \right) \in ]A_1A_2[,
\end{equation}
where
$$A_1 = \left( \frac{n+1}{2(n-1)},\frac{n+1}{2(n-1)} -\frac{2}{n} \right) \in P_2P_4 \text{   and   } A_2 = \left(  1,\frac{n-2}{n}\right) \in P_4P_5.$$
Observe that $]A_1A_2[ \in \Delta_{P_2P_4P_5} \setminus \Delta_{P_1P_2P_3}$ for $n\geqslant 4$.
This leads to $\left( 1-\dfrac{1}{r_0},1-\dfrac{1}{s}\right) = \left( \dfrac{1}{r_0'},\dfrac{1}{s'} \right) \in \Delta_{P_2P_4P_5}$.

Now, by using Tonelli’s theorem, the H\"older inequality and Proposition \ref{Yamazaki} for $(d_1, d_2) = (r_0',s')$ with noting that \eqref{d1d2},  we can estimate for each radial symmetric function $\phi\in C^\infty_{rad}(\mathbb{R}^n)$ and $f\in L^\infty(\mathbb{R},L^{(s,\infty)}_{rad}(\mathbb{R}^n))$ that
\begin{eqnarray}
|\left<\zeta(f)(t),\phi\right>| &=& \left|\int_{\mathbb{R}^n}\zeta(f)(t,x)\phi(x)dx \right|\cr
&\leqslant& \int_{-\infty}^{+\infty} \left<  |f(\tau,\cdot)|, |W(t-\tau)\phi|\right> d\tau\cr
&\leqslant& C\int_{-\infty}^{+\infty}\|f(\tau,\cdot) \|_{(s,\infty)} \| W(t-\tau)\phi\|_{(s',1)}d\tau\cr 
&\leqslant& C\sup_{t\in \mathbb{R}}\| f(t,\cdot)\|_{(s,\infty)} \int_{-\infty}^{+\infty} \| W(t-\tau)\phi\|_{(s',1)}d\tau\cr
&\leqslant& K\sup_{t\in \mathbb{R}}\| f(t,\cdot)\|_{(s,\infty)} \|\phi \|_{(r_0',1)}.
\end{eqnarray}
This shows that
\begin{equation}\label{bound1}
\|\zeta(f)(t) \|_{(r_0,\infty)} \leqslant K\sup_{t\in \mathbb{R}}\| f(t,\cdot)\|_{(s,\infty)}
\end{equation}
provided that $r_0 = \dfrac{n(p-1)}{2}$ and $s=\dfrac{r_0}{p}$.

Using inequality \eqref{bound1} and H\"oder inequality with noting that $\dfrac{1}{s} = \dfrac{1}{r_0}+\dfrac{2}{n} = \dfrac{b}{n}+\dfrac{q}{r_0}$, we can estimate the right hand-side of \eqref{InterEq} as follows
\begin{eqnarray}\label{Bounded}
\|\hbox{RHS of \eqref{InterEq}}\|_{(r_0,\infty)} &\leqslant& \| \dot{W}(t)u_0 + W(t)u_1\|_{(r_0,\infty)} \cr
&&+ \| \zeta(-V_1v)(t)\|_{(r_0,\infty)} + \| \zeta(V_2F(v))(t)\|_{(r_0,\infty)} \cr
&\leqslant& \| \dot{W}(t)u_0 + W(t)u_1\|_{(r_0,\infty)} + K\sup_{t\in \mathbb{R}}\| V_1v(t)\|_{(s,\infty)} + K\sup_{t\in \mathbb{R}}\| V_2F(v(t))\|_{(s,\infty)}\cr
&\leqslant& \| \dot{W}(t)u_0 + W(t)u_1\|_{(r_0,\infty)} + KC_0\| V_1\|_{(\frac{n}{2},\infty)}\sup_{t\in \mathbb{R}}\| v(t)\|_{(r_0,\infty)}\cr
&&+ KC\| V_2\|_{(\frac{n}{b},\infty)}\sup_{t\in \mathbb{R}}\||v(t)|^{q-1}v(t) \|_{(\frac{r_0}{q},\infty)}\cr
&\leqslant& \|(u_0,u_1) \|_{\mathcal{I}_{rad}}+ KC_0\| V_1\|_{(\frac{n}{2},\infty)}\sup_{t\in \mathbb{R}}\| v(t)\|_{(r_0,\infty)}\cr
&&+ KC\| V_2\|_{(\frac{n}{b},\infty)}\sup_{t\in \mathbb{R}}\|v(t)\|^q_{(r_0,\infty)}.
\end{eqnarray}
This boundedness leads to the existence of a solution $u \in L^\infty(\mathbb{R}, L^{(r_0,\infty)}_{rad}(\mathbb{R}^n))$ for the linear integral equation \eqref{InterEq}. The uniqueness holds clearly. Therefore, we can define the solution operator $\Phi: L^\infty(\mathbb{R}, L^{(r_0,\infty)}_{rad}(\mathbb{R}^n)) \to L^\infty(\mathbb{R}, L^{(r_0,\infty)}_{rad}(\mathbb{R}^n))$ as follows
\begin{equation}
\Phi(v)(t)=u(t),\,\,\, t\in \mathbb{R},
\end{equation}
which is solution of \eqref{InterEq}, i.e., we have
\begin{equation}\label{FormulaPhi}
\Phi(v)(t) = \dot{W}(t)u_0 + W(t)u_1 + \int_0^tW(t-s)(-V_1(x)v(s,x)+ V_2(x)F(v(s)))ds.
\end{equation}
In order to prove the well-posedness of equation \eqref{W2} we prove that the operator $\Phi$ acts the ball $B_\rho = \left\{ u \in L^\infty(\mathbb{R}, L^{(r_0,\infty)}_{rad}(\mathbb{R}^n)): \, \sup\limits_{t\in \mathbb{R}}\| u(t)\|_{(r_0,\infty)}\leqslant \rho \right\}$ into itself and is a contraction for $\rho$ and $\| (u_0,u_1)\|_{\mathcal{I}_{rad}}$ are small enough. Indeed, from inequality \eqref{Bounded} and formula \eqref{FormulaPhi} we have
for each $v\in B_\rho$ that
\begin{eqnarray}
\| \Phi(v)(t)\|_{(r_0,\infty)} &\leqslant& \| (u_0,u_1)\|_{\mathcal{I}_{rad}} + KC_0 \| V_1\|_{(\frac{n}{2},\infty)}\rho + KC\|V_2 \|_{(\frac{n}{b},\infty)}\rho^q \cr
&\leqslant& \rho
\end{eqnarray}
provided that $\rho$, $\| V_1\|_{(\frac{n}{2},\infty)}$ (hence $c_1$) and $\| (u_0,u_1)\|_{\mathcal{I}_{rad}}$ are small enough. Therefore, the mapping $\Phi$ acts $B_\rho$ into itself.

Now, for $v_1,v_2\in B_\rho$, we can estimate by the same way as \eqref{Bounded} that
\begin{eqnarray}\label{Lipschitz}
\|\Phi(v_1)(t) - \Phi(v_2)(t) \|_{(r_0,\infty)} &\leqslant& \| \zeta(-V_1)(v_1(t)-v_2(t))\|_{(r_0,\infty)} + \| \zeta(V_2(F(v_1)-F(v_2))(t)\|_{(r_0,\infty)}\cr
&\leqslant& KC_0 \| V_1\|_{(\frac{n}{2},\infty)} \sup_{t\in \mathbb{R}}\| v_1(t) - v_2(t)\|_{(r_0,\infty)} \cr
&&+ KC\| V_2\|_{(\frac{n}{b},\infty)}\rho^{q-1}\sup_{t\in \mathbb{R}}\| v_1(t) - v_2(t) \|_{(r_0,\infty)}\cr
&=& \left( KC_0 \| V_1\|_{(\frac{n}{2},\infty)} +  KC\| V_2\|_{(\frac{n}{b},\infty)}\rho^{q-1}\right) \sup_{t\in \mathbb{R}}\| v_1(t) - v_2(t)\|_{(r_0,\infty)}.
\end{eqnarray} 
Hence, the mapping $\Phi$ is contraction provided that $\| V_1\|_{(\frac{n}{2},\infty)}$ (hence $c_1$) and $\rho$ are small enough such that
\begin{equation}\label{Condition}
0<KC_0 \| V_1\|_{(\frac{n}{2},\infty)} +  KC\| V_2\|_{(\frac{n}{b},\infty)}\rho^{q-1}<1.
\end{equation}
By using fixed point arguments, for the above values of $\rho$, $c_1$, and $\norm{(u_0,u_1)}{\mathcal{I}{rad}}$, we obtain that $\Phi$ has a fixed point $u$, which is also a solution of equation \eqref{intergralEq}, and hence is a mild solution of equation \eqref{Klein}.

{The uniqueness holds from inequalities \eqref{Lipschitz} and \eqref{Condition}. Indeed, if $u$ and $\hat{u}$ are two mild solutions of equation \eqref{Klein} in the ball $B_\rho$ corresponding to the initial data $(u_0,u_1)$, then formula \eqref{FormulaPhi} yields}
{\begin{eqnarray}
u(t)-\hat{u}(t) &=& \int_0^tW(t-s)(-V_1(x)u(s,x)+ V_2(x)F(u(s,x))ds \cr
&&- \int_0^tW(t-s)(-V_1(x)\hat{u}(s,x)+ V_2(x)F(\hat{u}(s,x))ds\cr
&=&\Phi(u)(t) - \Phi(\hat{u})(t)
\end{eqnarray}
Using inequality \eqref{Lipschitz}, we estimate
\begin{eqnarray}\label{Unique}
\sup_{t\in \mathbb{R}}\norm{u(t)-\hat{u}(t)}_{(r_0,\infty)} &=& \sup_{t\in \mathbb{R}}\norm{\Phi(u)(t) - \Phi(\hat{u})(t)}_{(r_0,\infty)}\cr
&\leqslant& \left( KC_0 \| V_1\|_{(\frac{n}{2},\infty)} +  KC\| V_2\|_{(\frac{n}{b},\infty)}\rho^{q-1}\right) \sup_{t\in \mathbb{R}}\| u(t) - \hat{u}(t)\|_{(r_0,\infty)}.
\end{eqnarray} 
This is equivalent to
\begin{equation}\label{Unique}
\left( 1-  KC_0 \| V_1\|_{(\frac{n}{2},\infty)} - KC\| V_2\|_{(\frac{n}{b},\infty)}\rho^{q-1} \right)\sup_{t\in \mathbb{R}}\| u(t) - \hat{u}(t)\|_{(r_0,\infty)} \leqslant 0.
\end{equation}
Using \eqref{Condition}, we obtain from \eqref{Unique} that
$\sup\limits_{t\in \mathbb{R}} \|u(t)-\hat{u}(t)\| = 0,$
which implies that $u(t) = \hat{u}(t)$ for all $t\in \mathbb{R}$. 
Hence, the uniqueness of the mild solution to equation \eqref{Klein} in the ball $B_\rho$ follows.} \\

\item[$(ii)$] We show only the property (\ref{ine16}). The proof of
(\ref{ine17}) is left to the reader. We start by defining
\[
u_{1}^{+}=u_{1}+\int_{0}^{+\infty}W(-s)\left( -V_1u(s)+V_2F(u(s)) \right)ds\text{ and }\,u_{0}^{+}=u_{0}.
\]
For $t>0$, consider $u^{+}$ given by
\[
u^{+}=\dot{W}(t)u_{0}+W(t)u_{1}+\int_{0}^{+\infty}W(t-s)\left( -V_1u(s) + V_2F(u(s)) \right)ds,
\]
and note that
\begin{equation}\label{Scat1}
u(t)-u^+(t) = \int_{t}^{+\infty}W(s-t) \left(-V_1u(s) + V_2F(u(s)) \right)ds.
\end{equation}
Setting 
\begin{equation}
\gamma(f)(t,x) = \int_{t}^{+\infty}W(s-t) f(s,x)ds.
\end{equation}
Similar to Assertion $(i)$ we can prove that
\begin{equation}\label{gamma}
\|\gamma(f)(t) \|_{(r_0,\infty)} \leqslant \tilde{K}\sup_{t\in \mathbb{R}}\| f(t,
\cdot)\|_{(s,\infty)}
\end{equation}
for $r_0 = \dfrac{n(p-1)}{2}$ and $s = \dfrac{r_0}{p}$. Applying inequality \eqref{gamma} to \eqref{Scat1}, we obtain that
\begin{eqnarray}\label{Scat2}
\| u(t)-u^+(t)\|_{(r_0,\infty)} &=& \left\| \int_{t}^{+\infty}W(s-t) \left(-V_1u(s) + V_2F(u(s)) \right)ds\right\|_{(r_0,\infty)}\cr
&\leqslant& \| \gamma(-V_1u)(t)\|_{(r_0,\infty)} + \| \gamma(V_2F(u))(t) \|_{(r_0,\infty)}\cr
&\leqslant& \tilde{K}\tilde{C}_0\| V_1\|_{(\frac{n}{2},\infty)}\sup_{t\in \mathbb{R}}\| u(t)\|_{(r_0,\infty)} + \tilde{K}\tilde{C}\| V_2\|_{(\frac{n}{b},\infty)}\sup_{t\in \mathbb{R}}\| u(t)\|^q_{(r_0,\infty)}\cr
&\leqslant& \tilde{K}\tilde{C}_0\| V_1\|_{(\frac{n}{2},\infty)}\rho + \tilde{K}\tilde{C}\| V_2\|_{(\frac{n}{b},\infty)}\rho^q\cr
&<& +\infty
\end{eqnarray}
for all $t>0$. This convergence shows that
\begin{equation}
\lim_{t\to + \infty}\| u(t)-u^+(t)\|_{(r_0,\infty)} = \lim_{t\to +\infty}\left\| \int_{t}^{+\infty}W(s-t) \left(-V_1u(s) + V_2F(u(s)) \right)ds\right\|_{(r_0,\infty)}=0
\end{equation}
which leads to the scattering behaviour \eqref{ine16}. Our proof is completed.
\end{proof}
\begin{remark}
If we consider $b=0$, then we have $p=q$ and the proof of Theorem \ref{Global} is still valid (see also \cite[Theorem 3.3]{Fe2017} for this case).
\end{remark}
\subsection{Polynomial stability}\label{S32}
In this part we establish a polynomial stability of global mild solution obtained in Theorem \ref{Global}.
\begin{theorem}\label{polynomial}(Polynomial stability). Suppose that $\tilde{u} \in L^\infty(\mathbb{R},L^{(r_0,\infty)}_{rad}(\mathbb{R}^n))$ is other mild solution of equation \eqref{Klein} corresponding with initial data $(\tilde{u}%
_{0},\tilde{u}_{1})$. Then, for a positive constant $0 < h<1$, we have
\begin{equation}
\lim_{|t|\rightarrow\infty}|t|^{h}\left\Vert \dot{W}(t)(u_{0}%
-\tilde{u}_{0})+W(t)(u_{1}-\tilde{u}_{1})\right\Vert _{(r_0,\infty)}%
=0\label{condition}%
\end{equation}
if and only if
\begin{equation}
\lim_{|t|\rightarrow\infty}|t|^{h}\left\Vert u(t)-\tilde
{u}(t)\right\Vert _{(r_0,\infty)}=0.\label{Stability}
\end{equation}
\end{theorem}
\begin{proof}
Without loss of generality, assume also that $t>0$. There exist positive constants $\rho$ and $\tilde{\rho}$ such that $\| u(t)\|_{(r_0,\infty)}<\rho$ and $\| \tilde{u}(t)\|_{(r_0,\infty)}<\tilde{\rho}$ for all $t$.
Then, we can estimate
\begin{eqnarray}\label{Stab1}
t^h\left\Vert u(t)-\tilde{u}(t)\right\Vert _{(r_0,\infty)}  &\leqslant&
t^h\left\Vert \dot{W}(t)(u_{0}-\tilde{u}_{0})+W(t)(u_{1}-\tilde
{u}_{1})\right\Vert _{(r_0,\infty)} \cr
&&+ t^h\left\| \int_0^t W(t-\tau) [-V_1(u(\tau)-\tilde{u}(\tau))]d\tau\right\|_{(r_0,\infty)} \cr
&&+ t^h\left\|\int_0^t W(t-\tau)[V_2(F(u)(\tau)-F(\tilde{u})(\tau))]d\tau\right\|_{(r_0,\infty)}\cr
&\leqslant&
t^h\left\Vert \dot{W}(t)(u_{0}-\tilde{u}_{0})+W(t)(u_{1}-\tilde
{u}_{1})\right\Vert _{(r_0,\infty)} \cr
&&+ t^h \int_0^t \left\| W(t-\tau) [-V_1(u(\tau)-\tilde{u}(\tau))] \right\|_{(r_0,\infty)} d\tau \cr
&&+ t^h \int_0^t \left\| W(t-\tau)[V_2(F(u)(\tau)-F(\tilde{u})(\tau))] \right\|_{(r_0,\infty)}d\tau.
\end{eqnarray}
Now we prove that
\begin{equation}\label{decay}
t^h \int_0^t \| W(t-\tau) f(\tau,x)\|_{(r_0,\infty)}d\tau \leqslant \tilde{L} \sup_{t\in \mathbb{R}_+} t^h \| f(t,\cdot)\|_{s,\infty}.
\end{equation}
Indeed, we split the left hand-side into two parts
\begin{equation}
I_1 = t^h\int_0^{t/2} \| W(t-\tau) f(\tau,x)\|_{(r_0,\infty)}d\tau \text{   and   } I_2=t^h \int_{t/2}^t \| W(t-\tau) f(\tau,x)\|_{(r_0,\infty)}d\tau. 
\end{equation}
Using the dispersive estimate \eqref{W5} the first part can be estimated as follows
\begin{eqnarray}\label{decay1}
|\left<I_1,\phi\right>| &\leqslant& t^h\int_0^{t/2} \| f(\tau,x)\|_{(s,\infty)}\| W(t-\tau)\phi\|_{(s',1)}d\tau\cr
&\leqslant& t^h\int_0^{t/2} (t-\tau)^{-n\left( \frac{1}{s'} - \frac{1}{r_0'}\right)+1} \| f(\tau,x)\|_{(s,\infty)}\| \phi\|_{(r_0',1)}d\tau\cr
&\leqslant& t^h \int_0^{t/2} (t-\tau)^{-1}\tau^{-h} d\tau \sup_{t\in \mathbb{R}_+}(t^h\| f(t,\cdot)\|_{(s,\infty)}) \|\phi \|_{(r_0',1)} \text{   (we used \eqref{d1d2})}\cr
&\leqslant& L_1 \sup_{t\in \mathbb{R}_+}(t^h\| f(t,\cdot)\|_{(s,\infty)}) \|\phi \|_{(r_0',1)}
\end{eqnarray}
for each $\phi \in C_{rad}^\infty(\mathbb{R}^n)$. Using Proposition \ref{Yamazaki} we estimate the second part as 
\begin{eqnarray}\label{decay2}
|\left<I_2,\phi\right>| &\leqslant& t^h\int_{t/2}^{t} (t-\tau)^{n\left( \frac{1}{s'} - \frac{1}{r_0'}\right)-2} \| f(\tau,x)\|_{(s,\infty)}\| W(t-\tau)\phi\|_{(s',1)}d\tau \text{   (because \eqref{d1d2})}\cr
&\leqslant& t^h\int_{t/2}^{t} (t-\tau)^{n\left( \frac{1}{s'} - \frac{1}{r_0'}\right)-2} \tau^{-h}(\tau^h\| f(\tau,x)\|_{(s,\infty)})\| W(t-\tau)\phi\|_{(s',1)}d\tau\cr
&\leqslant& t^h (\frac{t}{2})^{-h}\int_{t/2}^t (t-\tau)^{n\left( \frac{1}{s'} - \frac{1}{r_0'}\right)-2} \| W(t-\tau)\phi\|_{(s',1)} d\tau \sup_{t\in\mathbb{R}_+}(t^h\| f(t,\cdot)\|_{(s,\infty)}) \| \phi\|_{(r_0',1)}\cr
&\leqslant& L_2\sup_{t\in\mathbb{R}_+}(t^h\| f(t,\cdot)\|_{(s,\infty)}) \| \phi\|_{(r_0',1)} \text{   (we used Proposition \ref{Yamazaki})}
\end{eqnarray} 
for each $\phi \in C_{rad}^\infty(\mathbb{R}^n)$. The inequalities \eqref{decay1} and \eqref{decay2} implies \eqref{decay}.

Applying this inequality to \eqref{Stab1} and using the H\"older inequality we obtain that
\begin{eqnarray}
t^h\left\Vert u(t)-\tilde{u}(t)\right\Vert _{(r_0,\infty)}  &\leqslant& t^h\left\Vert \dot{W}(t)(u_{0}-\tilde{u}_{0})+W(t)(u_{1}-\tilde
{u}_{1})\right\Vert _{(r_0,\infty)} \cr
&&+ \| V_1\|_{(\frac{n}{2},\infty)}\sup_{t\in \mathbb{R}}\left(t^h\| u(t) - \tilde{u}(t)\|_{(r_0,\infty)}\right) \cr
&&+ \| V_2\|_{(\frac{n}{b},\infty)}\left( \rho^{q-1} + \tilde{\rho}^{q-1}  \right)\left(\sup_{t\in \mathbb{R}}t^h\|u(t)-\tilde{u}(t)\|_{(r_0,\infty)}\right).
\end{eqnarray}
Therefore,
\begin{eqnarray}
&&\left( 1- \| V_1\|_{(\frac{n}{2},\infty)} -\| V_2\|_{(\frac{n}{b},\infty)}\left( \rho^{q-1} + \tilde{\rho}^{q-1}  \right)  \right)\sup_{t\in \mathbb{R}_+} \left(t^h\left\Vert u(t)-\tilde{u}(t)\right\Vert _{(r_0,\infty)}\right) \cr
&\leqslant& \sup_{t\in \mathbb{R}_+}\left( t^h\left\Vert \dot{W}(t)(u_{0}-\tilde{u}_{0})+W(t)(u_{1}-\tilde
{u}_{1})\right\Vert _{(r_0,\infty)}\right).
\end{eqnarray}
Observe that $1- \| V_1\|_{(\frac{n}{2},\infty)} -\| V_2\|_{(\frac{n}{b},\infty)}\left( \rho^{q-1} + \tilde{\rho}^{q-1}  \right)  >0$ if $\| V_1\|_{(\frac{n}{2},\infty)}$ (hence $c_1$), $\rho$ and $\tilde{\rho}$ are small enough.
This shows that: if the limit \eqref{condition} holds, then we have
$$\lim_{t\rightarrow +\infty}t^{h}\left\Vert u(t)-\tilde
{u}(t)\right\Vert _{(r_0,\infty)}=0.$$

Now, we consider that the limit \eqref{Stability} holds. Clearly, we have
\begin{eqnarray}
&&t^h\left\Vert \dot{W}(t)(u_{0}-\tilde{u}_{0})+W(t)(u_{1}-\tilde
{u}_{1})\right\Vert _{(r_0,\infty)}\cr &\leqslant& t^h\left\Vert u(t)-\tilde{u}(t)\right\Vert _{(r_0,\infty)} + t^h\left\| \int_0^t W(t-\tau) [-V_1(u(\tau)-\tilde{u}(\tau))]d\tau\right\|_{(r_0,\infty)} \cr
&&+ t^h\left\|\int_0^t W(t-\tau)[V_2(F(u)(\tau)-F(\tilde{u})(\tau))]d\tau\right\|_{(r_0,\infty)}\cr
&\leqslant& \sup_{t\in \mathbb{R}_+}\left(t^h\left\Vert u(t)-\tilde{u}(t)\right\Vert _{(r_0,\infty)} \right) \left( 1 + \| V_1\|_{(\frac{n}{2},\infty)} + \| V_2\|_{(\frac{n}{b},\infty)}\left( \rho^{q-1} + \tilde{\rho}^{q-1}  \right) \right)\cr
&\longrightarrow& 0
\end{eqnarray}
as $t$ tends to infinity by \eqref{Stability}. Therefore, the limit \eqref{condition} holds. Our proof is complete.
\end{proof}
\begin{remark}\label{rem}
%We notice that the limit \eqref{condition} holds if we consider that $(D(u_0-\tilde{u}_0,u_1-\tilde{u}_1) \in L^{(\frac{r_0}{r_0-1},\infty)}(\mathbb{R}^n)\times L^{(\frac{r_0}{r_0-1},\infty)}(\mathbb{R}^n)$ for $2<r_0 = \dfrac{n(p-1)}{2}\leqslant \frac{2(n+1)}{n-1}$ which is equivalent to $1<q \leqslant 1 +\dfrac{2(n+1)(2-b)}{n(n-1)}$. In this case, by using the $L^{r_0}-L^{r_0'}$-estimate \eqref{LpLp'1} of the wave operator $W(t)$ we find that 
%\begin{eqnarray}
%&&\|\dot{W}(t)(u_0-\tilde{u}_0) + W(t)(u_1-\tilde{u}_1)\|_{(r_0,\infty)}\cr 
%&\leqslant& \|W(t)D(u_0-\tilde{u}_0)\|_{(r_0,\infty)} + \|W(t)(u_1-\tilde{u}_1)\|_{(r_0,\infty)} \cr
%&\leqslant& Lt^{1-n\left( 1 - \frac{2}{r_0} \right)}\| D(u_0 - \tilde{u}_0) \|_{(\frac{r_0}{r_0 -1},\infty)} +  \| u_1 - \tilde{u}_1 \|_{(\frac{r_0}{r_0 -1},\infty)}
%\end{eqnarray}
%which provides $h= -1 + n\left( 1 - \dfrac{2}{r_0} \right)>0$.
\begin{itemize}
\noindent

\item[$(i)$] By using polynomial stability obtained in Theorem \ref{polynomial} we can improve the decay of scattering as: letting $u^\pm$ as in the proof of Assertion $(ii)$ of Theorem \ref{Global}, we have for $0<h<1$ that
\begin{equation}\label{improve}
\| u(t) - u^\pm (t)\|_{(r_0,\infty)} = O(|t|^{-h}), \text{   as   } t\to \pm\infty
\end{equation}
provided that
\begin{equation}
\lim_{|t|\to +\infty} |t|^h\| \dot{W}(t)u_0 + W(t)u_1\|_{(r_0,\infty)} = 0.
\end{equation}
Indeed, the proof of \eqref{improve} is done by improving \eqref{Scat2} with noting that $\lim\limits_{|t|\to +\infty}|t|^h\| u(t)\|_{(r_0,\infty)} =0$ (i.e., we used Theorem \ref{polynomial} for $\tilde{u}=0$).

\item[$(ii)$] We notice that the limit \eqref{condition} holds if we consider that $(D(u_0-\tilde{u}_0),u_1-\tilde{u}_1) \in L^{(r_0,\infty)}(\mathbb{R}^n)\times L^{(r_0,\infty)}(\mathbb{R}^n)$. Indeed, using the H\"older inequality and dispersive estimate \eqref{W5} we have for each $\phi\in C^\infty_{rad}(\mathbb{R}^n)$ that
\begin{eqnarray}
|\left<W(t)(u_1-\tilde{u}_1), \phi \right>| &=& |\left< u_1-\tilde{u}_1,W(t)\phi\right>|\cr
&\leqslant& \| u_1 - \tilde{u}_1\|_{(r_0,\infty)} \|W(t)\phi \|_{(r_0',1)}\cr
&\leqslant& M_2|t|^{-n\left( \frac{1}{s'} - \frac{1}{r_0'} \right)+1}\| u_1 - \tilde{u}_1\|_{(r_0,\infty)} \|\phi \|_{(s',1)}\cr
&\leqslant& M_2 |t|^{-1}\| u_1 - \tilde{u}_1\|_{(r_0,\infty)} \|\phi \|_{(s',1)} \text{   (we used \eqref{d1d2})}.
\end{eqnarray}
Hence, we obtain 
\begin{equation}\label{lim1}
|t|^h\| W(t)(u_1-\tilde{u}_1)\|_{(r_0,\infty)} \leqslant M_2 |t|^{-1+h}\| u_1-\tilde{u}_1\|_{(r_0,\infty)}.
\end{equation}
By the same way we can estimate 
\begin{equation}\label{lim2}
|t|^h\| \dot{W}(t)(u_0-\tilde{u}_0)\|_{(r_0,\infty)} \leqslant M_2 |t|^{-1+h}\| D(u_0-\tilde{u}_0) \|_{(r_0,\infty)}
\end{equation}
The inequalities \eqref{lim1} and \eqref{lim2} and condition $(D(u_0-\tilde{u}_0),u_1-\tilde{u}_1) \in L^{(r_0,\infty)}(\mathbb{R}^n)\times L^{(r_0,\infty)}(\mathbb{R}^n)$ guarantee that the limit \eqref{condition} holds as $|t|$ tends to infinity.

\end{itemize}
\end{remark}

\vspace{0.3cm}

\noindent\textbf{Data availability statement:} No data was created and used in this research.\\
\noindent\textbf{Conflict of interest statement:} The author declares no conflict of interest.

\end{document}